\documentclass[12pt,reqno]{amsart}
\usepackage{enumerate, latexsym, amsmath, amsfonts, amssymb, amsthm, graphicx, color}
\def\pmod #1{\ ({\rm{mod}}\ #1)}
\def\Z{\mathbb Z}
\def\N{\mathbb N}

\def\l{\left}
\def\r{\right}
\def\bg{\bigg}
\def\({\bg(}
\def\){\bg)}
\def\t{\text}
\def\f{\frac}

\def\ls{\leqslant}
\def\gs{\geqslant}

\def\sm{\setminus}

\def\ve{\varepsilon}

\def\eq{\equiv}

\def\da{\delta}

\def\Proof{\noindent{\it Proof}}

\theoremstyle{plain}
\newtheorem{theorem}{Theorem}

\newtheorem{lemma}{Lemma}

\theoremstyle{remark}
\newtheorem{remark}{Remark}
\theoremstyle{example}

\makeatletter
\@namedef{subjclassname@2020}{%
  \textup{2020} Mathematics Subject Classification}
\makeatother
 \vspace{4mm}

\begin{document}

\hbox{Preprint, {\tt arXiv:2105.03416}}
\medskip

\title[On partitions of integers with restrictions involving squares]{On partitions of integers with restrictions involving squares}

 \author[C. Huang]{Chao Huang}
\address{(Chao Huang) Department of Mathematics\\ Nanjing University\\
Nanjing 210093, People's Republic of China}
\email{dg1921004@smail.nju.edu.cn}

\author[Z.-W. Sun]{Zhi-Wei Sun$^\star$}
\address{(Zhi-Wei Sun, corresponding author) Department of Mathematics\\ Nanjing University\\
Nanjing 210093, People's Republic of China}
\email{zwsun@nju.edu.cn}

\begin{abstract}
In this paper, we study partitions of positive integers with restrictions involving squares.
We mainly establish the following two results (which were conjectured by Sun in 2013):

(i) Each positive integer $n$ can be written as $n=x+y+z$ with $x,y,z$ positive integers such that
$x^2+y^2+z^2$ is a square, unless $n$ has the form $n=2^{a}3^{b}$ or $2^{a}7$ with $a$ and $b$ nonnegative integers.

(ii) Each integer $n>7$ with $n\not=11,14,17$ can be written as $n=x+y+2z$ with $x,y,z$ positive integers such that $x^2+y^2+2z^2$ is a square.
\end{abstract}

\subjclass[2020]{Primary 11P83, 11E25;  Secondary 11D09.}

\keywords{Partitions, sums of squares, ternary quadratic forms, Diophantine equations.
\newline \indent
$^\star$ Corresponding author, supported by the Natural Science Foundation of China (grant no. 11971222).}

\maketitle

\section{Introduction}

A partition of $n\in\Z^+=\{1,2,3,\ldots\}$ is a way to write $n$ as a sum of positive integers (with repetitions allowed). Partitions of positive integers were first studied by Euler, and they play important roles in number theory and combinatorics.

Lagrange's four-square theorem states that each $n\in\N=\{0,1,2,\ldots\}$ can be written as $x^2+y^2+z^2+w^2$ with $x,y,z,w\in\N$.
Z.-W. Sun \cite{S1} refined this classical theorem in various ways and posed many conjectures
on sums of four squares with certain restrictions involving squares.
For example, Sun's 1-3-5 conjecture \cite{S1} states that any $n\in\N$ can be written as $x^2+y^2+z^2+w^2$
($x,y,z,w\in\N$) with $x+3y+5z$ a square, this was recently confirmed by Machiavelo and Tsopanidis
\cite{M1} via Hamilton quaternions.

Motivated by the refinements of Lagrange's four-square theorem, in this paper we study partitions of positive integers with certain restrictions involving squares.

   Now we state our main results.

\begin{theorem}\label{x+11y+13z} Let $n>2$ be an integer.

{\rm (i)} We can write $n=x+y+z$ with $x,y,z\in\Z^+$ such that $x+11y+13z$ is a square.

{\rm (ii)} We can write $n=x+y+z$ with $x,y,z\in\Z^+$ such that $x+240y+720z$ is a square.
\end{theorem}

\begin{theorem} \label{p^m} Let $a,b,c,m\in\Z^+$ with $a<b\ls c$ and $\gcd(b-a,c-a)=1$.
Then, any sufficiently large integer can be written as $x+y+z$ with $x,y,z\in\Z^+$ such that $ax+by+cz=p^m$ for some prime number $p$.
\end{theorem}
\begin{remark}  By P. Dusart [Du, Section 4], for $x\gs 3275$ there is a prime $p$ with $x\ls p\ls x+x/(2\log^2x)$. With the aid of this, we can modify
our proof of Theorem \ref{p^m} to show the following results:

(i) Any integer $n\gs6$ can be written as $x+y+z\ (x,y,z\in\Z^+)$ with $x+3y+6z=p^2$ for some prime $p$.

(ii) Any integer $n>6$ can be written as $x+y+z\ (x,y,z\in\Z^+)$ with $x+2y+7z=p^3$ for some prime $p$.

(iii) Any integer $n>12$ can be written as $x+y+z\ (x,y,z\in\Z^+)$ with $x+2y+10z=p^4$ for some prime $p$.
\end{remark}

Our third and fourth theorems were originally conjectured by Sun \cite{S13,S-13} in 2013.

\begin{theorem}\label{Th1.1} Let $n$ be a positive integer.  We can write $n=x+y+z$ with $x,y,z\in\Z^+$ such that $x^2+y^2+z^2$ is a square, if and only if $n$
  is neither of the form $2^a3^b\ (a,b\in\N)$ nor of the form $2^a7\ (a\in\N)$.
\end{theorem}
\begin{remark} This was stated as a conjecture by Sun in \cite[A230121]{S-13}. For example,
\begin{align*}
5=&1+2+2 \ \t{with} \ 1^2+2^2+2^2=3^2,\\
13=&1+4+8 \ \t{with} \ 1^2+4^2+8^2=9^2,\\
17=&2+9+6 \ \t{with} \ 2^2+6^2+9^2=11^2.
\end{align*}
\end{remark}

\begin{theorem}\label{Th1.2} Any integer $n>7$ with $n\not=11,14,17$
can be written as $n=x+y+2z$ with $x,y,z \in \Z^+$ such that $x^2+y^2+2z^2$ is a square.
\end{theorem}

\begin{remark} For each positive integer $n$, let $a(n)$ denote the number of ways to write
$n$ as $x+y+2z$ with $x,y,z\in\Z^+$ and $x\ls y$ such that $x^2+y^2+2z^2$ is a square.
The sequence $a(n)\ (n=1,2,3,\ldots)$ is available from \cite[A230747]{S-13}. In particular,
$a(n)=1$ for $n=9,21,34,56$. Note that
\begin{align*}9 =& 1 + 4 + 2\times2\ \t{with}\ 1^2 + 4^2 + 2\times2^2 = 5^2,
\\21 =& 5 + 8 + 2\times4\ \t{with}\ 5^2 + 8^2 + 2\times4^2 = 11^2,
\\34 =&7 + 25 + 2\times1\ \t{with}\ 7^2 + 25^2 + 2\times1^2 = 26^2,
\\56 =& 14 + 14 + 2\times14\ \t{with}\ 14^2 + 14^2 + 2\times14^2 = 28^2.
\end{align*}
\end{remark}

\begin{theorem}\label{Th1.3}
Let $k\geq 4$ be an integer. Then any integer $n> \max\{20k,1200\}$ can be written as the
$x_1+\cdots+x_k$ with $x_1,\ldots,x_k\in\Z^+$ such that $x_1^2+\cdots+x_k^2$ is a square.
\end{theorem}

We are going to prove Theorems 1.1-1.2 in the next section.
Theorems 1.3-1.5 will be proved in Sections 3-5 respectively.

\section{Proofs of Theorems 1.1-1.2}

\begin{lemma}\label{ax+by} Let $a,b\in\Z^+$ with $\gcd(a,b)=1$. Then any integer $n>ab$
can be written as $ax+by$ with $x,y\in\Z^+$.
\end{lemma}
\Proof. It is known that any integer $m>ab-a-b$ can be written as $au+bv$ with $u,v\in\N$.
As $n-a-b>ab-a-b$, there are $u,v\in\N$ with $au+bv=n-a-b$, and hence $n=ax+by$ with $x=u+1\in\Z^+$ and $y=v+1\in\Z^+$.
This concludes the proof. \qed

\medskip\noindent {\it Proof of Theorem} 1.1. (i) The result can be verified directly for $n=3,4,\ldots,30$.

Now let $n\in\N$ with $n\gs 31$. Choose $a\in\{\lfloor\sqrt n\rfloor+5,\lfloor\sqrt n\rfloor+6\}$ with $a\eq n\pmod 2$.
Since
$$\l(\sqrt n-\f35\r)^2\gs(\sqrt{31}-0.6)^2>\l(\f 35\r)^2+3.6,$$
we have $n-(6/5)\sqrt n>3.6$, i.e., $10n-12\sqrt n>36$. Therefore
$$a^2\ls(\sqrt n+6)^2=n+12\sqrt n+36<11n.$$
On the other hand,
$$a^2>(\sqrt n+4)^2=n+8\sqrt n+16$$
and hence
$$\f{a^2-n}2>4\sqrt n+8\gs 4\sqrt{31}+8>30=5\times6.$$
By Lemma \ref{ax+by}, there are $y,z\in\Z^+$ with $5y+6z=(a^2-n)/2$. Since
$$5(y+z)<5y+6z=\f{a^2-n}2<\f{11n-n}2=5n,$$
we have $y+z<n$. Hence $x=n-y-z\in\Z^+$ and
$$x+11y+13z=n+10y+12z=n+2(5y+6z)=n+2\times\f{a^2-n}2=a^2.$$
This concludes the proof of part (i) of Theorem 1.1.

(ii) For $n=3,4,\ldots,722$ we can easily verify the desired result via a computer.

Below we fix $n\in\N$ with $n\gs723$. Let $a=\lfloor \sqrt n\rfloor+k$ with $k=390$. Then
\begin{align*} a^2-n>&(\sqrt n+k-1)^2-n=2(k-1)\sqrt n+(k-1)^2
\\\gs& 2(k-1)\sqrt{723}+(k-1)^2>239\times719=171841.
\end{align*}
In view of Lemma \ref{ax+by}, we have $a^2-n=239y+719z$ for some $y,z\in\Z^+$. Note that
$$\l(\sqrt{239n}-\f k{\sqrt{239}}\r)^2\gs\l(\sqrt{239\times723}-\f k{\sqrt{239}}\r)^2\gs\f{240}{239}k^2-479$$
and hence $239n-2k\sqrt n\gs k^2-479.$ Thus
$$239y+719z=a^2-n\ls(\sqrt n+k)^2-n=k^2+2k\sqrt n\ls 239n+479$$
and hence $239(y+z)\ls 239n+479-(719-239)z<239n$. Therefore $x=n-y-z\in\Z^+$ and
$$x+240y+720z=n+239y+719z=n+(a^2-n)=a^2.$$
This completes our proof. \qed

\medskip\noindent {\it Proof of Theorem} 1.2. For $x>1$ let $\pi(x)$ denote the number of primes not exceeding $x$.
Let $\ve>0$. By the Prime Number Theorem,
$$\pi((1+\ve)x)-\pi(x)\ \sim\ \ve \f x{\log x}$$
as $x\to+\infty$. So, if $x$ is large enough then there is a prime $p$ with $x<p\ls(1+\ve)x$.

Observe that
$$\lim_{n\to+\infty}\f{(bn)^{1/m}}{(an+(b-a)(c-a))^{1/m}}=\l(\f ba\r)^{1/m}>1.$$
By the above, there is a positive integer $N$ such that for any integer $n\gs N$ there is a prime $p$ for which
$$(an+(b-a)(c-a))^{1/m}<p\ls (bn)^{1/m}$$
and hence
$$(b-a)(c-a)<p^m-an<(b-a)n.$$
As $\gcd(b-a,c-a)=1$, by Lemma \ref{ax+by} there are positive integers $y$ and $z$ such that
$$(b-a)(y+z)\ls (b-a)y+(c-a)z=p^m-an<(b-a)n.$$
Thus $x:=n-y-z\in\Z^+$ and $ax+by+cz=p^m$.

The proof of Theorem 1.2 is now complete. \qed

\section{Proof of Theorem 1.3}

\setcounter{lemma}{0}
\setcounter{remark}{0}
\setcounter{equation}{0}

For convenience, we set $\square=\{x^2:\ x\in\Z\}$.

\begin{lemma}
Let $n$ be a positive integer with $n, n/6, n/7\not\in \square$. Suppose that the equation
\begin{equation} \label{2.1} n=x^2+y^2-3z^2\ (x,y,z\in\Z)
\end{equation}
has solutions. Then, there are $x_0,y_0,z_0\in\Z^+$ with $x_0^2+y_0^2-3z_0^2=n$
satisfying
\begin{equation}  x_0\geq z_0>0\ \t{and}\ \  y_0\geq 2z_0.
\end{equation}
Moreover, we may require $x_0> z_0$ if $n=x^2-2z^2$ for some $x,z\in\Z^+$
with $x/z \in (2,\,3.5]\cup (5,+\infty).$
\end{lemma}
\begin{proof}
If $n=x^2+y^2$ with $x,y\in\N$,  then we may assume $x \geq y >0$ since $n\not\in \square$. Thus $n=x^2+(2y)^2-3y^2$, whence $(x_0,y_0,z_0)=(x,2y,y)$ meets (3.2).

Now assume that $n$ is not a sum of two squares. Choose a particular solution $(r,s,t)$ of \eqref{2.1}
with $r,s\in\N$ and
$$t=\min\{z\in\Z^+:\ n=x^2+y^2-3z^2\ \t{for some}\ x,y\in\N\}.$$
 In view of the identity $a^2-3b^2=(3b-2a)^2-3(a-2b)^2$, the equation \eqref{2.1}
has three other solutions:
\begin{align}
&(r, 3t-2s, 2t-s),\\
&(3t-2r,s, 2t-r),\\
&(3t-2r, 2s+3r-6t, 4t-2r-s).
\end{align}
By the definition of $t$, we get $|2t-s|\geq t$ from the solution in (3.3). So we have either $s \leq t $ or $ s \geq  3t$. Similarly, by the solution in (3.4),  either $r \leq t $ or $ r\geq  3t$.
Since $r^2+s^2-3t^2=n$, one of $r$ and $s$ is greater than $t$ and hence at least $3t$.
If $r\gs 3t$ and $s\gs 3t$, then $(x_0,y_0,z_0)=(r,s,t)$ satisfies (3.2).

Now we handle the case $r\leq t$ and $s\geq 3t$.
(The case $s\le t$ and $r\ge 3t$ can be handled similarly.)

Suppose $s < 5t-2r$. Then
$$-t <  4t-2r-s \leq t-2r \leq t.$$
By the definition of $t$ and the solution (3.5), we must have $|4t-2r-s|=t$ and hence
 $4t-2r-s = t-2r = t$. So $r=0$ and $s=3t$. It follows that $n=r^2+s^2-3t^2=6t^2$
 which contradicts $n/6\not\in\square$.

By the last paragraph, we must have
$s \geq 5t-2r$. Note that the solution
$$(x_0,y_0,z_0)=(3t-2r, s, 2t-r)$$
satisfies (3.2) since
$$s \geq 2(2t-r),  \ \  3t-2r \geq 2t-r\ \t{and}\ \  2t-r\geq t>0.$$

In view of the above, we have proved the first assertion of Lemma 3.1.

Now we prove the second assertion in Lemma 3.1. Suppose that $n=x^2-2z^2$ for some $x,z\in \Z^+$ with $x/z \in (2,3.5]\cup (5,+\infty)$. As $n/7\not\in \square$, we have $x/z\not=3$.
We want to find a solution $(x_0,y_0,z_0)$ of (3.1) satisfying $(3.2)$ and the inequality $x_0>z_0$.

{\it Case} 1. $x/z\in(2,3)$, i.e., $0<2z<x<3z.$

In this case, $(x_0,y_0,z_0)= (z, 2x-3z, x-2z)$ meets our purpose since
\begin{align*}
x^2-2z^2&= (z)^2+(2x-3z)^2-3(x-2z)^2, \\
 x_0-z_0&= z-(x-2z)= 3z-x >0,\\
 y_0-2z_0&=(2x-3z)-2(x-2z)=z>0.
 \end{align*}

{\it Case} 2. $x/z\in(3,3,5)$, i.e., $0< 3z < x \leq 3.5z$.

Using the identity
$$n=x^2-2z^2= (3x-8z)^2+(2x-3z)^2-3(2x-5z)^2, $$
we find that $(x_0,y_0,z_0)=(3x-8z, 2x-3z, 2x-5z)$ meets our purpose as
\begin{align*}
 x_0-z_0&=(3x-8z)-(2x-5z)=x-3z>0,\\
 y_0-2z_0&=(2x-3z)-2(2x-5z)=7z-2x \geq 0.
\end{align*}

{\it Case} 3. $x/z\in(5,6)$, i.e., $5z < x< 6z$.

In this case,
$$n=x^2-2z^2=(2x-9z)^2+(5z)^2-3(6z-x)^2$$
and hence $(x_0,y_0,z_0)=(2x-9z,5z,6z-x)$ meets our purpose.

{\it Case} 4. $x/z\in[6,+\infty)$, i.e., $x\gs 6z$.
In this case,
$$n=x^2-2z^2=(5z)^2+x^2-3(3z)^2$$
and hence $(x_0,y_0,z_0)=(5z,x,3z)$ meets our purpose.

In view of the above, we have completed the proof of Lemma 3.1.

\end{proof}

\begin{lemma} {\rm (\cite[p.\, 164]{D1})}
Let $p$ be an odd prime with $p\not\eq 1 \pmod{24}$. Let $F(x,y,z)$ be any classic, indefinite, anisotropic ternary quadratic form with determinant $-p$ . Then
\begin{align*}&\Z\sm\{F(x,y,z):\ x,y,z\in\Z\}
\\=&\{4^k(8l+p):\ k\in\N,\ l\in\Z\}
\\&\cup\l\{p^{2k+1}(pl+r^2):\ k\in\N,\ l\in\Z,\ 1\ls r\ls\f{p-1}2\r\}.
\end{align*}
\end{lemma}

\begin{remark}
The reader may consult \cite{Ro} for a more general result.
\end{remark}

\medskip
\noindent {\it Proof of Theorem 1.3}.
(i) We first prove the ``if" direction.
Let $p$ be the smallest prime divisor of $n$. Then $p>3$ and $p\not=7$. Write $n=pq$ with $q\in\Z^+$.
If $p=x+y+z$ for some $x,y,z\in\Z^+$ with $x^2+y^2+z^2\in\square$, then
$n=qx+qy+qz$ and $(qx)^2+(qy)^2+(qz)^2=q^2(x^2+y^2+z^2)\in\square$.

By the last paragraph, it suffices to consider only the case in which $n$ is an odd prime with $n\not=3,7$. We need to find $x,y,z\in\Z^+$ such that
$$x+y+z=n\ \t{and}\ x^2+y^2+z^2 \in \square.$$

If $a,b,c,d$ are integers with
\begin{equation}\label{abcd1} 2n=(a+c+d)^2+(b+c-d)^2-3c^2-3d^2,
\end{equation}
then, for
\begin{equation}\label{xyz} x=\f{a^2+b^2-c^2-d^2}2,\ y=ac-bd,\ z=ad+bc,
\end{equation}
we have $x+y+z=n$ and
$$x^2+y^2+z^2=x^2+(a^2+b^2)(c^2+d^2)=\l(\f{a^2+b^2+c^2+d^2}2\r)^2.$$
So, it suffices to find $a,b,c,d\in\Z$ satisfying \eqref{abcd1} such that $x,y,z$
given by \eqref{xyz} are positive.

As $2n$ is neither of the form $3^{2u+1}(3v+1)\ (u,v\in\N)$
nor of the form $4^{u}(8v+3)\ (u,v\in\N)$,
in view of Lemma 3.2 we have $2n\in\{x^2+y^2-3z^2:\ x,y,z\in\Z\}$.
By Lemma 3.1, there are integers $x_0,y_0,z_0\in\Z^+$ with $2n= x_0^2+y_0^2-3z_0^2$
 for which $x_0 \geq 2z_0$ and $y_0\geq z_0$; moreover, we may require $y_0>z_0$
 if $2n=r^2-2s^2$ for some $r,s\in\Z^+$ with $r/s\in(2,3.5]\cup[5,+\infty)$.

{\it Case} 1. $y_0>z_0$.

In this case, we set
$$  a=x_0-z_0, \ b=y_0-z_0, \ c=z_0, \ d=0. $$
It is easy to see that \eqref{abcd1} holds and $a\gs c >0, \ b>0$ so that $x,y,z$ given by \eqref{xyz} are positive.

{\it Case} 2. $y_0=z_0$.

In this case, $2n=x_0^2+y_0^2-3z_0^2=x_0^2-2z_0^2$ and $x_0/z_0\not\in(2,3.5]\cup(5,+\infty)$. It's clear $x_0/z_0=5$ contradicts the assumption that $n$ is a prime. Hence $x_0/z_0\in(3.5,5)$.

If $x_0/z_0\in(4,5)$, then it is easy to see that the integers
$$a=x_0-2z_0,\ b=2z_0\ \t{and}\ c=d=z_0$$
meet our purpose.

Now we assume that $3.5<x_0/z_0\ls 4$. If $x_0=4z_0$, then
$2n=x_0^2-2z_0^2=14z_0^2,$ which contradicts $n\neq 7$.
Thus $3.5z_0<x_0<4z_0$. Set
$$a=x_0-2z_0, \ b=5z_0-x_0, \ c=x_0-2z_0, \ d=z_0.$$
Then
$$a+c+d=2x_0-3z_0, \ b+c-d=2z_0, \ c=x_0-2z_0,$$
and hence \eqref{abcd1} holds. It is easy to see that $x>0$ and $z>0$.
Note also that
\begin{align*}
y=&ac-bd=(x_0-2z_0)^2-(5z_0-x_0)z_0
\\=& {x_0}^2-3x_0z_0-z_0^2=(x_0-1.5z_0)^2-3.25z_0^2
\\>&4z_0^2-3.25z_0^2>0.
\end{align*}
This concludes our proof of the ``if" direction.

(ii) Now we prove the ``only if" direction. If $n$ is even and $x,y,z$ are positive integers with
$x+y+z=n$ and $x^2+y^2+z^2\in\square$, then $x^2+y^2+z^2$ is a multiple of $4$
and hence none of $x,y,z$ is odd. Thus $n/2=x_0+y_0+z_0$ with $x_0^2+y_0^2+z_0^2\in\square$,
where $x_0=x/2,\ y_0=y/2,\ z_0=z/2$ are positive integers. So it remains to prove that
any $n\in\{7\}\cup\{3^b:\ b\in\N\}$ cannot be written as $x+y+z$ with $x,y,z\in\Z^+$ and $x^2+y^2+z^2\in\square$. It is easy to see that this holds for $n=3,7$.

Now assume $n=3^b$ for some integer $b\gs2$. Suppose that $n=x+y+z$ with $x,y,z\in\Z^+$ and $x^2+y^2+z^2\in\square$. If we don't have $x\eq y\eq z\pmod 3$, then exactly one of $x,y,z$
is divisible by $3$ since $x+y+z\eq0\pmod 3$, and hence $x^2+y^2+z^2\eq2\pmod 3$ which contradicts
$x^2+y^2+z^2\in\square$. Thus $x\eq y\eq z\eq \da\pmod3$ for some $\da\in\{0,1,2\}$.
 Write $x=3x'+\da$, $y=3y'+\da$ and $z=3z'+\da$ with $x',y',z'\in\Z$.
 Then $x'+y'+z'=n/3-\da\eq-\da\pmod 3$ and hence
 $$x^2+y^2+z^2\eq 6(x'+y'+z')\da+3\da^2\eq-6\da^2+3\da^2=-3\da^2\pmod 9.$$
 As $x^2+y^2+z^2$ is a square, we must have $\da=0$. Thus $n/3=x'+y'+z'$
 with $(x')^2+(y')^2+(z')^2=(x^2+y^2+z^2)/9\in\square$.
 Continuing this process, we finally get that $3$ can be written as $x+y+z$ with $x,y,z\in\Z^+$
 and $x^2+y^2+z^2\in\square$, which is absurd. This contradiction concludes our proof of the
 ``only if" direction.

 In view of the above, we have completed the proof of Theorem 1.3. \qed

\section{Proof of Theorem 1.4}
\setcounter{equation}{0}
\setcounter{example}{0}

\medskip
\noindent {\it Proof of Theorem 1.4}.
If $n=x+y+2z$ for some $x,y,z\in\Z^+$ with $x^2+y^2+2z^2\in\square$, then
$2n=2x+2y+2(2z)$ and $(2x)^2+(2y)^2+2(2z)^2=4(x^2+y^2+2z^2)\in\square$.
 So, without loss of generality,  we simply assume that $n$ is odd.
 For positive odd integer $n\ls1.5\times10^6$, we can verify the desired result via a computer.
Below we suppose that $n$ is odd and greater than $1.5\times10^6$.
 We need to find $x,y,z,w\in\Z^+$ such that
\begin{equation}n=x+y+2z \ \t{and}\ x^2+y^2+2z^2=w^2.
\end{equation}

Let $a$ and $c$ be positive odd integers. Define
\begin{equation}
b=\begin{cases}-1&\t{if}\ n+1-4ac=2,
\\|n+1-4ac|/2&\t{otherwise},\end{cases}
\end{equation}
and
\begin{equation}d=\begin{cases}-1&\t{if}\ n+1-4ac >2,
\\1&\t{otherwise}.\end{cases}
\end{equation}
Note that
\begin{equation}\label{abcd} n=4ac-2bd-d^2.
\end{equation}
Define
\begin{equation}\label{st} \ s= 4a^2-c^2+2b^2+2bd\ \ \t{and}\ \ t= 2bc+2ad+cd.
\end{equation}
Then
$$s\eq4-1+2b^2+2b\eq3\pmod4\ \ \t{and}\ \ t\eq cd\eq1\pmod 2.$$
Note that
\begin{equation}\label{xyzn}
x=\frac{n+(-1)^{(n+1)/2}s+2t}4,\ y= \frac{n+(-1)^{(n+1)/2}s-2t}4,\ z= \frac{n-(-1)^{(n+1)/2}s}4
\end{equation}
 are all integers.
 It is easy to verify that $(4.1)$ holds for such $x,y,z$ and
 $$w=2a^2+b^2+bd+(c^2+d^2)/2.$$

We claim that $x,y,z$ are positive provided that
\begin{equation}\label{abc<}
a\geq450, \ \  1.69a < c< 1.79a\ \ \t{and}\ \  |b|<0.658a.
\end{equation}
It is easy to see that $s\ge0$ and $t\ge0$.
By \eqref{abc<}, we have
$$4ac+c^2-4a^2-2b^2-4bc > 0.038a^2$$
and
$$|4bd+d^2+4ad+2cd| < 1+ 10.212a <10.22a <0.025a^2.$$
Combining these with \eqref{abcd} and \eqref{st} , we get
\begin{align*}
n-s-2t&=4ac-2bd-d^2-(4a^2-c^2+2b^2+2bd)-(4bc+4ad+2cd) \\
&\geq4ac+c^2-4a^2-2b^2-4bc -|4bd+d^2+4ad+2cd| >0.
\end{align*}
It follows that $x,y,z$ given by \eqref{xyzn} are positive.

Now it remains to find odd integers $a$ and $c$ satisfying \eqref{abc<}.
Choose $\da_1,\da_2\in\{0,1\}$ such that
\begin{equation}\label{ac0}
a_0 =\lfloor \sqrt{(n+1)/6.96}\rfloor +\delta_1
\ \ \t{and}\ \ c_0 = \lfloor \sqrt{1.74(n+1)/4} \rfloor+\delta_2
\end{equation}
are both odd. As $n>1.5\times10^6$, we have
\begin{align}\label{a0c0}
&a_0 \geq 465, \ \  c_0 \geq 807, \\
\label{c0/a0}
&  1.734< c_0/a_0 <1.747,\\
\label{a0-c0}
&16a_0-8c_0>2.024a_0,\\\label{11a0}
&|4a_0c_0-n-1| < 4(a_0+c_0)+4 <11a_0.
\end{align}

 If $|n+1-4a_0c_0|/2 < 0.658a_0$, then $(a,c)=(a_0,c_0)$ meets our purpose.

 Below we suppose $|n+1-4a_0c_0|/2\gs0.658a_0$. In light of \eqref{11a0},
 we may choose $m\in\{0,\pm1\}$
 with $|4a_0c_0-n-1-8ma_0|\leq|4a_0|$.
 Then, in view of \eqref{a0-c0}, we choose $k\in\{0,\pm1,\pm2\}$ such that
\begin{equation}\label{m}
 |4a_0c_0-n-1-8ma_0+k(16a_0-8c_0)|\leq 1.012a_0.
\end{equation}
 If $k=\pm2$, then we must have
  $$3.036a_0\leq |4a_0c_0- n-1-8ma_0|\leq 4a_0$$
   and hence we can choose $m=0$ first. Therefore  $|16km-32k^2|\leq 128$.

Clearly, \begin{equation}
a=a_0-2k \ \t{and}\ \  c=c_0-2m+4k,
\end{equation}
are odd integers with $a\geq 450$. Note also that
\begin{align*}|a-\sqrt{(n+1)/ 6.96}|\leq5
\ \ \t{and}\ \ |c-\sqrt{1.74(n+1)/4}|\leq9.
\end{align*}
Therefore, with the aids of (4.8) and (4.9), we get
\begin{align*}0.989\leq a / \sqrt{(n+1)/6.96}\leq 1.0116
\end{align*}
and
$$0.988\leq c /\sqrt{1.74(n+1)/4}\leq 1.0121.$$
Therefore, $1.69a \leq c\leq 1.79a$ as desired.

By \eqref{m}, we also have $|b|=|n+1-4ac|/2<0.658a$, since
\begin{align*}
|4ac-n-1|&= |4a_0c_0-n-1-8ma_0-8kc_0+16ka_0+16mk-32k^2| \notag \\
 & \leq |4a_0c_0-n-1-8ma_0+k(16a_0-8c_0)|+|16mk-32k^2| \notag \\
 &\leq 1.012a_0+128 \leq 1.012(a+4)+128 \leq1.316a.
\end{align*}
Thus \eqref{abc<} holds and this concludes our proof
of Theorem 1.4. \qed

Let us illustrate our proof of Theorem 1.4 by a concrete example.
\medskip

\noindent{\bf Example 3.1}. For $n=1,000,001$, we take $a_0=379$ and $c_0=659$ by \eqref{ac0}.
 Then $|n+1-4a_0c_0| =958 > 1.316a_0$. As in our proof of Theorem 1.4, we choose $m=0$ and $k=1$, and then get $a=377,b=99,c=663,d=-1$. Then $s=148351$ and $t=129857$ by \eqref{st}. This yields the solution
\begin{align*}
\begin{cases}277841+147984+2\times287088 =1000001 ,\\
277841^2+147984^2+2\times287088^2=513745^2.\end{cases}
\end{align*}

\section{Proof of Theorem 1.5}
\setcounter{lemma}{0}
\setcounter{equation}{0}

\begin{lemma} {\rm (Cauchy's Lemma \cite[p.\,31]{Na})}
Let $a$ and $b$ be positive odd integers such that
\begin{equation}\label{ab}
b^2 <4a \ \ \t{and} \ \  3a<b^2+2b+4.
\end{equation}
Then there are $s,t,u,v\in\N$ such that
\begin{equation}
s+t+u+v=b\ \t{and}\ s^2+t^2+u^2+v^2=a.
\end{equation}
\end{lemma}

\begin{lemma} Let $m$ and $n$ be positive odd integers with $3m^2< n^2 <4m^2 $.
Then there are $s_0,t_0,u_0,v_0\in\Z^+$ such that
$$s_0+t_0+u_0+v_0=n \  \t{and} \ \  s_0^2+t_0^2+u_0^2+v_0^2=m^2.$$
\end{lemma}
\begin{proof}
Let $a=m^2-2n+4$ and $b=n-4$. Then \eqref{ab} holds.
By Lemma 5.1, there are $s,t,u,v\in\N$
satisfy (5.2). Define
\[ s_0=s+1,\ t_0=t+1,\ u_0=u+1,\ v_0=v+1.  \]
Then
$$ s_0+t_0+u_0+v_0 =b+4=n$$
and
$$ s_0^2+t_0^2+u_0^2+v_0^2 = a+ 2b+4 =m^2.$$
This concludes the proof. \end{proof}

\medskip
\noindent {\it Proof of Theorem 1.5}. Clearly, it suffices to consider only the case $2\nmid n$ with $n>\max\{10k,600\}$.

Let $j=k-4$ and consider the interval $I=(n/4+7j/2, \ n/3+10j/3)$.
Suppose that $I$ contains no odd square. Then, for some $h\in\Z$ we have
\[ (2h-1)^2 \leq   \frac n 4+\frac {7j} 2 <\frac n 3+ \frac {10j} 3 \leq (2h+1)^2 \]
and hence
\[ 4h=(2h+1)^2-(2h-1)^2 > \frac n {12}-\frac j 6 > \frac n {15} >40, \]
which implies  $h>10$. Thus
\[ \frac n 4 + \frac {7j} 2 \geq (2h-1)^2 > 19(2h-1)> 36h > 9\l(\frac n {12} - \frac j 6\r)\]
and hence $10j> n$, which contradicts our assumption.

By the above, there exists odd integer $m$ such that
\begin{equation} \frac n 4+\frac {7j} 2 < m^2 < \frac n 3+ \frac {10j} 3,
\end{equation}
and hence
 $$3(m^2-4j) < n-2j < 4(m^2-4j).$$
 By Lemma 5.2, there are $x_1,x_2,x_3,x_4\in\Z^+$
such that
$$x_1+x_2+x_3+x_4=n-2j\ \ \t{and}\ \ x_1^2+x_2^2+x_3^2+x_4^2=m^2-4j.$$
Set $x_i=2$ for $4<i\ls k$. Then $\sum_{i=1}^kx_i=n$ and
$$\sum_{i=1}^kx_i^2=m^2-4j+j\times2^2=m^2.$$

In view of the above, we have completed the proof of Theorem 1.5. \qed

 \end{document}